\documentclass[10pt,a4paper]{amsart}
\usepackage[utf8]{inputenc}
\usepackage[pagebackref,colorlinks,linkcolor=red,citecolor=blue,urlcolor=blue,hypertexnames=false]{hyperref}
\usepackage{csquotes}

\usepackage{amsmath,amsthm,amssymb,amsopn}
\usepackage{amsrefs}
\usepackage{amscd}
\input xy
\xyoption{all}
\newdir{ >}{{}*!/-5pt/@{>}}
\usepackage{mathtools}
\let\overfence\overbrace 
\let\downfencefill\downbracefill 
\patchcmd{\overfence}{\downbracefill}{\downfencefill}{}{}
\patchcmd{\downfencefill}{\braceru \bracelu}{}{}{}

\usepackage{etoolbox}
\apptocmd{\sloppy}{\hbadness 10000\relax}{}{}

\newcommand{\comment}[1]{} 
	
\usepackage{tikz-cd}

\def\Z{\mathbb{Z}}

\def\C{\mathbb{C}}

\def\topdf{\texorpdfstring}
\theoremstyle{plain}
\newtheorem{thm}[equation]{Theorem}
\newtheorem{lem}[equation]{Lemma}
\newtheorem{coro}[equation]{Corollary} 
\newtheorem{prop}[equation]{Proposition}

\theoremstyle{definition}

\newtheorem{question}{Question}

\newtheorem{ex}[equation]{Example}

\theoremstyle{remark} 

 \newtheorem{rem}[equation]{Remark}

  \numberwithin{equation}{section}
\newtheorem*{ack}{Acknowledgements}

\newcommand{\cK}{\mathcal K}
\newcommand{\cL}{\mathcal L}

\newcommand{\cO}{\mathcal O}

\newcommand{\cR}{\mathcal R}

\newcommand{\cU}{\mathcal U}

\def\fA{\mathfrak{A}}
\def\fB{\mathfrak{B}}
\def\fC{\mathfrak{C}}

\def\fF{\mathfrak{F}}

\def\fI{\mathfrak{I}}
\def\fJ{\mathfrak{J}}

\newcommand{\BF}{\fB\fF}

\def\ab{\mathfrak{Ab}}

\def\leav{\operatorname{Leavitt}}

\def\can{\operatorname{can}}

\def\vf{\operatorname{vf}}

\def\alg{\mathrm{Alg}}

\def\sotimes{\overset{\sim}{\otimes}}
\newcommand{\aha}{{\alg_{\C}}}
\newcommand{\ahas}{{{\rm Alg}^*_{\C}}}
\newcommand{\lra}{\longrightarrow}
\newcommand{\iso}{\overset{\sim}{\lra}}

\newcommand{\onto}{\twoheadrightarrow}

\def\triqui{\vartriangleleft}

\def\inc{\operatorname{inc}}

\def\ad{\operatorname{ad}}

\def\id{\operatorname{id}}

\newcommand{\coker}{{\rm Coker}}
\renewcommand{\ker}{{\rm Ker}}
\newcommand{\im}{\mathrm{Im}}

\newcommand{\op}{\mathrm{op}}

\DeclareMathOperator*{\colim}{colim}

\def\comp{\operatorname{comp}}
\def\compu{\widetilde{\comp}}
\renewcommand{\top}{\operatorname{top}}
\begin{document}
\hfuzz=22pt
\vfuzz=22pt
\hbadness=2000
\vbadness=\maxdimen

\author{Guillermo Cortiñas}
\thanks{Supported by CONICET and partially supported by grants PICT 2017-1395 from Agencia Nacional de Promoci\'on Cient\'\i fica y T\'ecnica, UBACyT 0150BA from Universidad de Buenos Aires, and PGC2018-096446-B-C21 from Ministerio de Ciencia e Innovaci\'on}
\title{Lifting graph $C^*$-algebra maps to Leavitt path algebra maps}

\begin{abstract} 
Let $\xi:C^*(E)\to C^*(F)$ be a unital $*$-homomorphism between simple purely infinite Cuntz-Krieger algebras of finite graphs. We prove that there exists a unital $*$-homomorphism $\phi:L(E)\to L(F)$ between the corresponding Leavitt path-algebras such that $\xi$ is homotopic to the map $\hat{\phi}:C^*(E)\to C^*(F)$ induced by completion. We show moreover that $\hat{\phi}$ is a homotopy equivalence in the $C^*$-algebraic sense if and only if $\phi$ is a homotopy equivalence in the algebraic, polynomial sense. We deduce, in particular, that any isomorphism between simple purely infinite Cuntz-Krieger algebras is homotopic to the completion of a unital algebraic homotopy equivalence. 
\end{abstract}

\maketitle

\section{Introduction}\label{sec:intro}

A finite, regular, directed graph consists of finite sets $E^0$ and $E^1$ of vertices and edges and maps $r,s:E^1\to E^0$
with $s$ surjective. We write $L(E)$ for the Leavitt path algebra \cite{lpabook} of $E$ over $\C$, which is a $*$-algebra, and $C^*(E)$ for its $C^*$-completion. We say  that $E$ is \emph{simple purely infinite}, or \emph{spi}, if it satisfies the hypothesis of the purely infinite simplicity theorem \cite{lpabook}*{Theorem 3.1.10}, which by that theorem are equivalent to the purely infinite simplicity of $L(E)$ and also to that of $C^*(E)$ (see \cite{lpabook}*{Section 5.6}). A theorem of Cuntz and R\o rdam \cite{ror}*{Theorem 6.5}, which is now a particular case of the Kirchberg-Phillips theorem \cite{chris}*{Theorem 4.1.1 and Corollary 4.2.2}, says that the group $K_0(C^*(E))$ scaled by the class $[C^*(E)]$ of its free module of rank one, is a complete invariant for the isomorphism class of the $C^*$-algebra of a finite spi graph $E$. The classification question for Leavitt path algebras \cite{quest}, restricted to $\C$-algebras, asks whether a similar result holds for the Leavitt path algebra $L(E)$. Since the scaled groups $(K_0(C^*(E)),[C^*(E)])$ and $(K_0(L(E)),[L(E)])$ are isomorphic, the above question can be restated as follows:
\begin{question}
For finite spi graphs $E$ and $F$, does the existence of a $*$-isomorphism $C^*(E)\cong C^*(F)$ imply that of an algebra isomorphism $L(E)\cong L(F)$?
\end{question}
By \cite{chris}*{Theorem 4.1.1 and Corollary 4.2.2}, one may equivalently substitute unital $C^*$-homotopy equivalence for $C^*$-algebra isomorphism in the question above. By \cite{cm2}*{Theorem 6.1}, $C^*(E)\cong C^*(F)$ implies that $L(E)$ and $L(F)$ are homotopy equivalent in the algebraic sense; by \cite{classinvo}*{Theorem 14.1} one can further conclude that an involution preserving homotopy equivalence exists upon hyperbolic matricial stabilization. A main result of the current paper is the following.

\begin{thm}\label{intro:semimain} Let $E$ and $F$ be finite spi graphs and let $\xi:C^*(E)\to C^*(F)$ be a unital $*$-homomorphism. Then there exists a unital $*$-homomorphism $L(E)\to L(F)$ whose completion $\hat{\phi}:C^*(E)\to C^*(F)$ is $C^*$-homotopic to $\xi$. Moreover if $\xi$ is a homotopy equivalence in the $C^*$-algebra sense, then $\phi$ is a homotopy equivalence in the algebraic, polynomial sense. 
\end{thm}

We point out that in the theorem above, algebraic homotopy equivalence is understood in the sense of not necessarily involution preserving algebra homomorphisms. Thus the $*$-homomorphism $\phi$ of the theorem is a homotopy equivalence if there exists a --not necessarily involution preserving-- algebra homomorphism $\psi:L(F)\to L(E)$ such that the composites $\psi\circ\phi$ and $\phi\circ\psi$ are algebraically homotopic to the respective identity maps.

The main tools we use to prove the theorem above are Kasparov's bivariant $K$-theory of $C^*$-algebras and the algebraic bivariant $K$-theory introduced in \cite{ct}. The latter consists of a triangulated category $kk$ and a functor $j:\aha\to kk$ from the category of algebras, which is algebraically homotopy invariant, maps algebra extensions to distinguished triangles and is matricially stable, and is universal initial with these properties. Similarly, Kasparov's $K$-theory can also be described as consisting of triangulated category $KK$ and a functor $k:C^*-\alg\to KK$ from the category of separable $C^*$-algebras, which is homotopy invariant and stable in the $C^*$-algebra sense, and maps those extensions admitting completely positive splittings to distinguished triangles, and is universal initial with these properties \cite{ralf}. Write $\ahas$ for the category of $*$-algebras and $*$-homomorphisms and 
\[
\leav^*\subset\ahas \text{ and } \langle\leav\rangle_{kk}\subset kk
\]
for the full subcategories on the Leavitt path algebras of finite regular graphs.

A key step in proving Theorem \ref{intro:semimain} above is the following, which is also a main result of the article.

\begin{thm}\label{intro:onto}
Let $\widehat{\ \ }:\leav^*\to C^*-\alg$ be the completion functor. There is a $\Z$-linear, full and conservative functor $\compu:\langle\leav\rangle_{kk}\to KK$ such that the following diagram commutes
\[
\xymatrix{\leav^*\ar[d]^j\ar[r]^{\widehat{\ \ }}& C^*-\alg\ar[d]^k\\
          \langle\leav\rangle_{kk}\ar[r]_(.6){\compu}& KK}
\]
\end{thm}

The rest of this article is organized as follows. In Section \ref{sec:extend} we consider, for a finite regular graph $E$ and a $C^*$-algebra $\fA$, the canonical bijection $\hom_{\ahas}(L(E),\fA)\iso\hom_{C^*-\alg}(C^*(E),\fA)$. Proposition \ref{prop:comp} extends the latter to a surjective natural transformation of functors $C^*-\alg\to\ab$,
\begin{equation}\label{intro:mapcompa}
\comp_{\fA}:kk(L(E),\fA)\to KK(C^*(E),\fA),
\end{equation}
which is an isomorphism in certain cases, for example when $\fA$ is either stable of properly infinite (\ref{ex:comp}).
In Section \ref{sec:functor} we prove Theorem \ref{thm:onto}, which contains Theorem \ref{intro:onto}. Observe that the prescription that the diagram of the latter theorem be commutative dictates the definition of $\compu$ on objects as $\compu(L(E))=C^*(E)$. To define $\compu$ on homomorphisms we compose the map
$kk(L(E),L(F))\to kk(L(E),C^*(F))$ with the natural transformation $\comp_{C^*(F)}$ of \eqref{intro:mapcompa}.

Two $C^*$-algebra homomorphisms $\phi,\psi:\fA\to\fB$ are \emph{$M_2$-homotopic} if they become $C^*$-homotopic upon composing with the inclusion $\iota_2:\fB\to M_2\fB$ into the upper left hand corner; algebraic $M_2$-homotopy between maps in $\aha$ is defined similarly. In Section \ref{sec:spi} we prove Theorem \ref{thm:semimain}, which contains Theorem \ref{intro:semimain}, and in addition says that completion sends $*$-homomorphisms satisfying a strong fullness assumption onto $M_2$-homotopy classes of nonzero $C^*$-algebra homomorphisms $C^*(E)\to C^*(F)$, and that $\hat{\phi}$ is an $M_2$-homotopy equivalence in $C^*-\alg$ if and only if $\phi$ is one in $\aha$. 
 
In Appendix \ref{sec:M2} we prove the technical Lemma \ref{lem:m2htpy} which says that if $\fA$ and $\fB$ are $C^*$-algebras with $\fB$ properly infinite, then the monoid of homotopy classes of $*$-homomorphisms from $\fA$ to the stable $C^*$-algebra $\cK\sotimes\fB$ is equivalent to the monoid of $M_2$-homotopy classes of $*$-homomorphisms $\fA\to\fB$. 

\begin{ack} The author wishes to thank Eusebio Gardella for a useful email exchange related to the subject of Appendix 
\ref{sec:M2}. Thanks also to Guido Arnone for his careful reading of the manuscript and helpful comments and suggestions.
\end{ack}

\section{Extending completion to a natural transformation \topdf{$kk(L(E),\fA)\to KK(C^*(E),\fA)$}{comp}}\label{sec:extend}

All algebras considered in this paper are over $\C$. If $A$ is an algebra, we write $K_*(A)$, $KV_*(\fA)$ and $KH_*(A)$ for its Quillen, Karoubi-Villamayor and Weibel's homotopy algebraic $K$-theory (\cite{kv1},\cite{kh},\cite{friendly}). If $\fA$ is a $C^*$-algebra, we further consider its Bott-periodic topological $K$-theory $K_*^{\top}(\fA)$. 

\begin{lem}\label{lem:factor}
Let $\fA$ be a $C^*$-algebra. For $n\in\Z$, let $K_n(\fA)\to K_n^{\top}(\fA)$ and $K_n(\fA)\to KH_n(\fA)$ be the canonical comparison maps. Then there is a natural map $KH_n(\fA)\to K_n^{\top}(\fA)$ making the following diagram commute.
\[
\xymatrix{K_n(\fA)\ar[dr]\ar[rr]&&K_n^{\top}(\fA)\\
&KH_n(\fA)\ar[ur]&}
\]
\end{lem}
\begin{proof} Here we use the notation for path and loop functors and the description of $KH$ of \cite{friendly}*{Sections 3--5}. Let $r+n\ge 0$. Restriction of polynomials to maps on the unit interval yields a map of extensions
\[
\xymatrix{0\to \Omega^{r+n+1}\fA\ar[d]\ar[r]&P^{r+n+1}\fA\ar[d]\ar[r]&\Omega^{r+n}\fA\ar[d]\to 0\\
          0\to\fA(0,1)^{r+n+1}\ar[r]&\fA(0,1]^{r+n+1}\ar[r]&\fA(0,1)^{r+n}\to 0.}
\]
The boundary map $K_{-r}^{\top}(\fA(0,1)^{r+n})\to K_{-r-1}^{\top}(\fA(0,1)^{r+n+1})$ is an isomorphism since $\fA(0,1]^{r+n+1}$ is contractible. Hence we have a natural map 
\[
KH_n(\fA)=\colim_rK_{-r}(\Omega^{n+r}\fA)\to \colim_rK_{-r}^{\top}(\fA(0,1)^{n+r})=K_n^{\top}(\fA).
\]                
For $n\le 0$, $K_n(\fA)$ is the first term of the inductive system in the colimit definining $KH_n(\fA)$ and the factorization of the lemma is clear. For $n\ge 1$ the map $K_n(\fA)\to KH_n(\fA)$ factors through Karoubi-Villamayor $K$-theory $KV_n(\fA)$ and the comparison maps fit in a commutative diagram
\[
\xymatrix{K_n(\fA)\ar[dr]\ar[r]&KV_{n}(\fA)=KV_1(\Omega^{n-1}\fA)\ar[d]\ar[r]& K_0(\Omega^n\fA)\ar[d]\\
                         &K_n^{\top}(\fA)=K_1^{\top}(\fA(0,1)^{n-1})& KH_n(\fA).\ar[l]}
\]
This finishes the proof. 
 \end{proof}

\begin{rem}\label{rem:factor}
The composite  $K_0(\fA)\to KH_0(\fA)\to K_0^{\top}(\fA)=K_0(\fA)$ is the identity map; thus $KH_0(\fA)\to K_0(\fA)$ is a split surjection.
\end{rem}

\begin{lem}\label{lem:khpi}
Let $\fA$ be a properly infinite $C^*$-algebra. Then $\fA$ is $K$-regular and the natural map $KH_*(\fA)\to K_*^{\top}(\fA)$ is an isomorphism. 
\end{lem}
\begin{proof}
We know from \cite{wicris}*{Theorem 3.2} that $K_*(\fA)\to K_*^{\top}(\fA)$ is an isomorphism, so by Lemma \ref{lem:factor} it suffices to show that $\fA$ is $K$-regular. This follows from \cite{wicris}*{Corollary 2.3} and the argument of \cite{roshand}*{Theorem 20}.
\end{proof}

In what follows we write $\fA\sotimes\fB$ for the \emph{minimal tensor product} of $C^*$-algebras $\fA$ and $\fB$. 
Let  $kk$ and $KK$ be the triangulated bivariant $K$-theory categories of $\C$-algebras and of $C^*$-algebras, and let $j:\aha\to kk$ (\cite{ct}) and $k:C^*-\alg\to KK$ \cite{ralf} be the canonical functors.

\begin{prop}\label{prop:comp}
Let $E$ be a regular finite graph. 
Then there is a surjective natural transformation of functors $C^*-\alg\to\ab$,
\[
\comp_{\fA}:kk(L(E),\fA)\onto KK(C^*(E),\fA),
\]
making the following diagram commute.
\[
\xymatrix{\hom_{\ahas}(L(E),\fA)\ar[r]^(.45){\sim}\ar[d]&\hom_{C^*-\alg}(C^*(E),\fA)\ar[d]\\
kk(L(E),\fA)\ar@{>>}[r]^{\comp_{\fA}}&KK(C^*(E),\fA)}
\]
If moreover $\fA$ is $K_0$-regular and $KV_1(\fA)\to K_1^{\top}(\fA)$ is injective, then $\comp_{\fA}$ is an isomorphism.
\end{prop} 
\begin{proof}
Let $E'$ be the essential graph obtained from $E$ by source removal \cite{lpabook}*{Definition 6.3.26}; there is a full projection $p\in L(E)$ such that $L(E')$ and $C^*(E')$ are $*$-isomorphic to $pL(E)p$ and $pC^*(E)p$. Hence we may assume that $E=E'$, by Morita invariance. Let $E_t$ be the transpose graph. The map $\comp_{\fA}$ is defined as the following composite, where the individual maps are explained below. 
\begin{multline*}
kk(L(E),\fA)\cong KH_1(\fA\otimes_\C L(E_t))\to
 KH_1(\fA\sotimes C^*(E_t))\to\\ K_1^{\top}(\fA\sotimes C^*(E_t))\cong KK(C^*(E),\fA).
\end{multline*}
The isomorphisms above come from Poincar\'e duality, proved for $KK$ and graph $C^*$-algebras in \cite{kamiput} and for $kk$ and Leavitt path algebras in \cite{classinvo}. The second map in the sequence is induced by the completion map $L(E_t)\to C^*(E_t)$ and the third is the comparison map from homotopy algebraic to topological $K$-theory, as in Lemma \ref{lem:factor}. Write $e_t\in E^1_t$ for the edge corresponding to an edge $e\in E^1$. By the proof of \cite{classinvo}*{Theorem 11.2}, the algebraic Poincar\'e dual of a $*$-homomorphism $\phi:L(E)\to\fA$ in $KH_1(L(E_t)\otimes\fA)$ is the class of the unitary $$1\otimes 1-\sum_{v\in E^0}\phi(v)\otimes v+\sum_{e\in E^1}\phi(e)\otimes e_t^*,$$ which maps to the class of the same unitary in $K_1^{\top}(\fA\sotimes C^*(E_t))$. This shows that the diagram of the proposition commutes, since the class just mentioned is also the image under $C^*$-Poincar\'e duality of the $*$-homomorphism $\hat{\phi}:C^*(E)\to \fA$ of \cite{kamiput}*{Section 4}. Next observe that passage to completion induces a map of extensions from the presentation of $L(E)$ as a quotient of the Cohn algebra \cite{lpabook}*{Definition 1.5.1 and Proposition 1.5.5} to that of $C^*(E)$ as a quotient of the Toeplitz algebra; this map in turn induces a map of exact sequences of abelian groups
\[
\xymatrix{
0\to \BF(E)^{\vee}\otimes KH_1(\fA)\ar@{->>}[d]\ar[r]&kk(L(E),\fA)\ar[d]^{\comp_{\fA}}\ar[r]&
\hom(\BF(E),KH_0(\fA))\ar@{->>}[d]\to 0\\
0\to \BF(E)^{\vee}\otimes K^{\top}_1(\fA)\ar[r]& KK(C^*(E),\fA)\ar[r]&\hom(\BF(E),K_0(\fA))\to 0}
\]
Here $\BF(E)^{\vee}$ is the dual Bowen-Franks group of \cite{classinvo}*{Section 12}, whose definition is recalled in \eqref{BF} below. The vertical map on the right is surjective by Remark \ref{rem:factor}; that of the left is surjective because $K_1(\fA)\to K_1^{\top}(\fA)$ is. Diagram chasing shows that $\comp_{\fA}$ is surjective and that it is an isomorphism whenever the map $KH_i(\fA)\to K_i^{\top}(\fA)$ is an isomorphism for $i=0,1$. If $\fA$ is $K_0$-regular then $KH_1(\fA)=KV_1(\fA)$, and the last assertion of the proposition follows, since $KV_1(\fA)\to K_1^{\top}(\fA)$ is always surjective.
\end{proof}

\begin{ex}\label{ex:comp} Jonathan Rosenberg has conjectured \cite{rosop}*{Conjecture 2.1} that any $C^*$-algebra is $K_0$-regular. It was shown in \cite{acta}*{Theorem 8.1} that commutative $C^*$-algebras are $K$-regular. If $\fA$ is a stable $C^*$-algebra, then $\fA$ is $K$-regular by \cite{rosop}*{Theorem 1.5} and the comparison map $K_*(\fA)\to K_*^{\top}(\fA)$ is an isomorphism by \cite{sw}*{Theorem 10.9} and \cite{kardisc}*{Th\'eor\`eme 4.9}. By Lemma \ref{lem:khpi} the same is true if $\fA$ is properly infinite.  Hence the map $\comp_{\fA}$ of Proposition \ref{prop:comp} is an isomorphism if $\fA$ is either stable or properly infinite. 
\end{ex}

\section{A functor from \topdf{$\langle \leav\rangle_{kk}$}{kk(leavitt)} to \topdf{$KK$}{KK}}\label{sec:functor}

A \emph{$*$-algebra} is a $\C$-algebra $A$ equipped with a semilinear involution $*:A\to A^{\op}$; a $*$-homomorphism of $*$-algebras is an involution preserving algebra homomorphism. Write $\ahas$ for the category of $*$-algebras and $*$-homomorphisms and 
\[
\leav^*\subset\ahas \text{ and } \langle\leav\rangle_{kk}\subset kk
\]
for the full subcategories on the Leavitt path algebras of finite regular graphs. If $E$ is a finite regular graph, we write $A_E$ for its incidence matrix and 
\begin{equation}\label{BF}
\BF(E)=\coker(I-A_E^t)\text{ and } \BF(E)^{\vee}=\coker(I-A_E)
\end{equation}
for its \emph{Bowen-Franks} and \emph{dual Bowen-Franks} groups.
\begin{thm}\label{thm:onto}
Let $\widehat{\ \ }:\leav^*\to C^*-\alg$ be the completion functor. Then there is a $\Z$-linear functor $\compu:\langle\leav\rangle_{kk}\to KK$ with the following properties.
\item[i)] The following diagram commutes
\[
\xymatrix{\leav^*\ar[d]^j\ar[r]^{\widehat{\ \ }}& C^*-\alg\ar[d]^k\\
          \langle\leav\rangle_{kk}\ar[r]_(.55){\compu}& KK}
\]
 \item[ii)] We have an exact sequence
\[
0\to \BF(E)^\vee\otimes\BF(F)\otimes \C^*\to kk(L(E),L(F))\overset{\compu}\lra KK(C^*(E),C^*(F))\to 0. 
\]
In particular, $\compu$ is a full functor.
\item[iii)] $\compu$ is a conservative functor. 
\end{thm}
\begin{proof}
In order that the diagram in i) commutes, we must set $\compu(L(E))=C^*(E)$; this defines $\compu$ on objects. To define it also on homomorphisms, let $\iota:kk(L(E),L(F))\to kk(L(E),C^*(F))$ be the map induced by the inclusion, let $\comp$ be as in Proposition \ref{prop:comp} and set
\[
\compu=\comp_{C^*(F)}\circ\iota:kk(L(E),L(F))\to KK(C^*(E),C^*(F)).
\]
It follows from Proposition \ref{prop:comp} that the diagram of part i) commutes. In particular $\compu$ preserves identity maps. Next we have to check that if $E$, $F$ and $G$ are finite regular graphs, $\xi\in kk(L(F),L(G))$ and $\eta\in kk(L(E), L(F))$, then 
\begin{equation}\label{eq:compufun}
\compu(\xi\circ\eta)=\compu(\xi)\circ\compu(\eta). 
\end{equation} 
First consider the case when $E$ and $F$ are essential graphs. Recall from \cite{actenso}*{Lemma 6.1}, that Leavitt path algebras of finite graphs are regular supercoherent, so the tensor product of two such algebras is $K$-regular by \cite{abc}*{Theorems 7.6 and 8.6}. Let $\omega(\eta)\in K_1(L(F)\otimes L(E_t))=KH_1(L(F)\otimes L(E_t))$ and $\omega(\xi)\in K_1(L(G)\otimes L(F_t))$ be the elements associated to $\eta$ and $\xi$ under the Poincaré duality isomorphism described in \cite{classinvo}*{Theorem 11.2}. Let $u_E=1-\sum_{v\in E^0}v\otimes v+\sum_{e\in E^1}e\otimes e_t^*\in K_1(L(E)\otimes L(E_t))$ and let $\rho_E:L(E_t)\otimes L(E)\to\Sigma_X$ be as in the proof of \cite{classinvo}*{Theorem 11.2}. One checks, using the explict formulas of \cite{classinvo}*{Theorem 11.2} for the Poincar\'e duality isomorphisms (which are in turn those of \cite{kamiput}) and the identification $K_n(\Sigma_XA)=K_{n-1}(A)$, that for the cup-product $\star$ of \cite{cv}*{Lemmas 4.5 and 8.3}, we have
\[
\omega(\xi\circ\eta)=(LG\otimes\rho_F\otimes L(F_t))(\omega(\xi)\star(LF\otimes\rho_E\otimes L(E_t))(\omega(\eta)\star u_E)).
\]
Since passage to the completion followed by the comparison map $K_*\to K^{\top}_*$ preserves cup-products, and since the formulas for the Poincar\'e duality isomorphism in the $C^*$-algebra setting are the same as in the algebraic one, we get the identity \eqref{eq:compufun}, under our current assumption that $E$ and $F$ are essential. Next we consider the general case. Let $E'$ be as in the proof of Proposition \ref{prop:comp} and let $\inc_{E}:L(E')\to L(E)$ and $\widehat{\inc}_E:C^*(E')\to C^*(E)$ be the full corner inclusions. It is straightforward from the definition of $\compu$ that 
\[
\compu(\eta)=\compu(\eta\circ\inc_{E})\circ(\widehat{\inc}_E)^{-1}.
\]
Hence using naturality and what we have just proved, we obtain
\begin{align*}
\compu(\xi)\circ\compu(\eta)&=\compu(\xi\circ\inc_{F})\circ(\widehat{\inc}_F)^{-1}\circ \compu(\eta\circ\inc_{E})\circ(\widehat{\inc}_E)^{-1}\\
&=\compu(\xi\circ\inc_{F})\circ\compu((\inc_F)^{-1}\circ\eta\circ\inc_E)\circ(\widehat{\inc}_E)^{-1}\\
&=\compu(\xi\circ\eta\circ\inc_E)\circ(\widehat{\inc}_E)^{-1}\\
&=\compu(\xi\circ\eta).
\end{align*}
Thus $\compu$ is a functor; this finishes the proof of part i). To prove part ii), observe that, by construction, we have a commutative diagram with exact rows
\[
\xymatrix{0\to \BF(E)^{\vee}\otimes K_1(L(F))\ar[r]\ar[d]& kk(L(E),L(F))\ar[d]^{\compu}\ar[r]& \hom(\BF(E),\BF(F))\ar@{=}[d]\to 0\\
0\to \BF(E)\otimes \ker(I-A_F^t)\ar[r]&KK(C^*(E),C^*(F))\ar[r]& \hom(\BF(E),\BF(F))\to 0.}
\]
The vertical map on the left is induced by the surjection in the following exact sequence
\[
0\to\BF(F)\otimes\C^*\to K_1(L(F))\to\ker(I-A_F^t)\to 0.
\]
This sequence splits because $\ker(I-A_F^t)$ is a free abelian group. Hence by the snake lemma, $\compu$ is surjective and  $\ker(\compu)=\BF(E)^\vee\otimes\BF(F)\otimes\C^*$, finishing the proof of part ii). It follows that for finite regular $E$, we have a surjective ring homomorphism
\begin{equation}\label{map:ontocompu}
\compu:kk(L(E),L(E))\onto KK(L(E),L(E)).
\end{equation}
To prove that $\compu$ is conservative it suffices to show that the kernel $\fJ$ of \eqref{map:ontocompu} is a nilpotent ideal. We will in fact show that $\fJ^2=0$. Let $C(E)$ be the Cohn algebra and $\pi:C(E)\to L(E)$ the projection. Then  
\begin{multline}\label{eq:compotriv}
\BF(E)\otimes\C^*=\ker(K_1(L(E))\to \ker(I-A_E^t))=\im(K_1(C(E))\overset{\pi}\lra K_1(L(E)))\\
\subset kk_1(\C,L(E)).
\end{multline}
Besides, by \cite{cm1}*{Formula (6.4)}, $\BF(E)^\vee=kk_{-1}(L(E),\C)$, and by \cite{cm1}*{Lemma 7.21}, 
$\fI$ is generated as an abelian group by the composites $\eta\circ\xi$ with 
\goodbreak
\noindent $\xi\in kk_{-1}(L(E),\C)$ and $\eta$ in \eqref{eq:compotriv}. Hence it suffices to show that $\xi\circ\eta=0$ for any such two elements. By definition, if $\eta$ is in \eqref{eq:compotriv}, then there is $\eta'\in kk_1(\C,C(E))$ such that $\eta=\pi\circ\eta'$. But by \cite{cm1}*{Theorem 4.2}, $C(E)$ is $kk$-isomorphic to $\C^{(E^0)}$, and so $\xi\circ\pi\in kk_{-1}(C(E),\C)=K_{-1}(\C)^{E^0}=0$. Thus $\xi\circ\eta=0$; this finishes the proof of the theorem.
\end{proof}

\begin{coro}\label{coro:onto}
If either $\BF(E)$ or $\BF(F)$ is finite, then $\compu:kk(L(E),L(F))\to KK(C^*(E),C^*(F))$ is an isomorphism.
\end{coro}
\begin{proof} Observe that $\BF(E)$ and $\BF(E)^{\vee}$ are finitely generated and have the same rank, so one is finite if and only if the other is. Now use part ii) of Theorem \ref{thm:onto} and the fact that $\C^*$ is a divisible group. 
\end{proof}

\section{Properly infinite and simple purely infinite algebras}\label{sec:spi}

In this section we consider algebraic (i.e. polynomial) homotopy between algebra homomorphisms and continuous, involution preserving homotopy between $C^*$-algebra homomorphisms; we write $\sim$ for the former and $\approx$ for the latter. Let $\iota_2:A\to M_2A$ be the upper left hand corner inclusion of an algebra into the $2\times 2$ matrices with entries in $A$. We say that two algebra homomorphisms $f,g:A\to B$ are \emph{$M_2$-homotopic}, and write $f\sim_{M_2} g$, if $\iota_2\circ f\sim \iota_2\circ g$. We put $[A,B]=\hom_{\alg_\C}(A,B)/\sim$ and $[A,B]_{M_2}=\hom_{\alg_\C}(A,B)/\sim_{M_2}$. The relation $\approx_{M_2}$ and the sets $[[\fA,\fB]]$ and $[[\fA,\fB]]_{M_2}$ of ($M_2$-) $C^*$-homotopy classes of $C^*$-algebra homomorphisms $\fA\to\fB$ are defined similarly.

Recall from \cite{black}*{Section 6.11} that an idempotent $p$ in a unital algebra $A$ is \emph{very full} if there are elements $x\in pA$ and $y\in Ap$ such that $yx=1$. A homomorphism of unital algebras $\phi:A\to B$ is \emph{very full} if $\phi(1)$ is a very full idempotent. We write $[L(E),A]_{M_2}^{\vf}$ for the set of $M_2$-homotopy classes of very full homomorphisms $L(E)\to A$. 

\begin{ex}\label{ex:vfull}
Any nonzero idempotent of a simple purely infinite (spi) unital ring is very full, and any idempotent $M_2$-homotopic to zero is itself zero. Hence if $E$ is a finite graph and $A$ an spi algebra, we have
\begin{equation}\label{eq:nozero1}
[L(E),A]^{\vf}_{M_2}=[L(E),A]_{M_2}\setminus\{0\}.
\end{equation}
If $E$ is spi and $A$ is any unital algebra containing elements $x_1,x_2,y_1,y_2$ such that $y_ix_j=\delta_{i,j}$, then by \cite{classinvo}*{Example 13.19}, the set $[L(E),A]^{\vf}_{M_2}$ is naturally a group, and there is a group isomorphism
\begin{equation}\label{map:homoleav}
[L(E),A]^{\vf}_{M_2}\cong kk(L(E),A).
\end{equation}
\end{ex}

Recall that a $*$-homomorphism of $C^*$-algebras $\phi:\fA\to\fB$ with $\fA$ unital is called \emph{full} 
if $\phi(1)$ is a full projection in $\fB$. We write $[[\fA,\fB]]^f\subset [[\fA,\fB]]$ and $[[\fA,\fB]]_{M_2}\subset [[\fA,\fB]]_{M_2}^f$ for the subsets of homotopy classes of full homomorphisms.  

\begin{ex}\label{ex:full*}
If $\fA$ and $\fB$ are separable and unital, and $\fA$ is also simple and nuclear, then 
by \cite{chris}*{Proposition 3.1.2 and Theorem 4.1.1}, and Lemma \ref{lem:m2htpy}, $[[\fA,\cO_\infty\sotimes\fB]]_{M_2}^f$ is a group under direct sum, and there is a group isomorphism
\begin{equation}\label{map:homopis}
[[\fA,\cO_\infty\sotimes\fB]]_{M_2}^f\cong KK(\fA,\fB).
\end{equation}
If in addition $\fB$ is simple, then the same is true of $\cO_\infty\sotimes\fB$, and thus
\begin{equation}\label{eq:nozero2}
[[\fA,\cO_\infty\sotimes\fB]]_{M_2}^f=[[\fA,\cO_\infty\sotimes\fB]]_{M_2}\setminus\{0\}.
\end{equation}
If furthermore $\fB$ is purely infinite, then $\cO_\infty\sotimes\fB\cong\fB$ by Kirchberg's theorem \cite{chris}*{Theorem 2.1.5} and any nonzero projection in $\fB$ is very full.
\end{ex}

\begin{prop}\label{prop:comp1}
Let $E$ be an spi graph and let $\fA$ be a unital separable $C^*$-algebra. Then there is a natural isomorphism
\begin{equation}\label{map:comp2}
[L(E),\cO_\infty\sotimes\fA]^{\vf}_{M_2}\iso [[C^*(E),\cO_\infty\sotimes\fA]]_{M_2}^f
\end{equation}
\end{prop}
\begin{proof} Properly infinite $C^*$-algebras are $K$-regular by Lemma \ref{lem:khpi}. In particular this applies to $\cO_\infty\sotimes\fA$, and thus \eqref{map:homoleav} provides an isomorphism
between the left hand side of \eqref{map:comp2} and $kk(L(E),\cO_\infty\sotimes\fA)$. The proof follows from this, together with Proposition \ref{prop:comp}, Lemma \ref{lem:khpi}, and \eqref{map:homopis}.
\end{proof}

\begin{coro}\label{coro:comp1}
Let $E$ and $F$ be finite spi graphs. Then there are
\item[i)] an isomorphism of abelian groups 
\[
[L(E),C^*(F)]_{M_2}\setminus\{0\}\iso [[C^*(E),C^*(F)]]_{M_2}\setminus\{0\}
\]
and
\item[ii)] an exact sequence
\[
0\to \BF(E)^{\vee}\otimes\BF(F)\otimes \C^*\to [L(E),L(F)]_{M_2}\setminus\{0\}\to [[C^*(E),C^*(F)]]_{M_2}\setminus\{0\}\to 0.
\]
\end{coro}
\begin{proof} Graph $C^*$-algebras of finite graphs are unital, separable and nuclear; moreover, because $F$ is simple purely infinite, the same is true of $C^*(F)$. Hence $[L(E),C^*(F)]_{M_2}\setminus\{0\}\cong [[C^*(E),C^*(F)]]_{M_2}\setminus\{0\}$ by  Proposition \ref{prop:comp1} together with Example \ref{ex:full*}, \eqref{eq:nozero1} and \eqref{eq:nozero2}. The exact sequence of the corollary follows from that of Theorem \ref{thm:onto} and the isomorphisms 
\eqref{map:homoleav} and \eqref{map:homopis}.
\end{proof}

We say that a $*$-homomorphism $\phi:A\to B$ of unital $*$-algebras has \emph{property (P)} if $\phi(1)B$ contains an isometry. Put
\[
\hom_{\ahas}(A,B)\supset \hom_{\ahas}(A,B)^{P}=\{\phi \text{ has property (P)}\}.
\]

\begin{thm}\label{thm:semimain}
Let $E$ and $F$ be finite, spi graphs. Then

\item[i)] The map $\hom_{\ahas}(L(E),L(F))^{P}\to [[C^*(E),C^*(F)]]_{M_2}\setminus\{0\}$, $\phi\mapsto [\hat{\phi}]$ is onto and maps the subset of unital $*$-homomorphisms onto the subset of homotopy classes of unital $C^*$-algebra homomorphisms.

\item[ii)] Let $\phi\in \hom_{\ahas}(L(E),L(F))^P$. Then $\hat{\phi}:C^*(E)\to C^*(F)$ is an $M_2$-continuous homotopy equivalence if and only if $\phi$ is a polynomial $M_2$-homotopy equivalence. If furthermore $\phi$ is unital, then $\hat{\phi}$ is a continuous homotopy equivalence if and only if $\phi$ is a polynomial homotopy equivalence.
\end{thm}
\begin{proof} By Example \ref{ex:full*}, we may replace $[[C^*(E),C^*(F)]]_{M_2}\setminus\{0\}$ by
$KK(C^*(E),C^*(F))$ in the statement of the theorem. Let $\xi\in KK(C^*(E),C^*(F))$ and let $\xi_0=K_0(\xi):\BF(E)\to \BF(F)$. Let $K_0^*$ be the Grothendieck group of projections as defined in \cite{classinvo}*{Section 7} and let $\can'_E:\BF(E)\to K_0(L(E))^*$ be the split monomorphism of \cite{classinvo}*{Remark 7.5}. By \cite{classinvo}*{Theorem 9.4} applied to $\can'_F\circ\xi_0$ there is a $*$-homomorphism $\phi:L(E)\to L(F)$ with property (P)
such that $K_0(\phi)^*\circ\can'_E=\can'_F\circ \xi_0$, which can be taken unital if $\xi_0$ is. Thus $K_0(\phi)=\xi_0$ by \cite{classinvo}*{Remark 7.5} and therefore 
\[
\eta:=\xi-k(\hat{\phi})\in \cL=\ker(K_0:KK(C^*(E),C^*(F))\to \hom(\BF(E),\BF(F))).
\] 
Let $L(F)_{\phi}=\bigoplus_{e\in E^1}\phi(ee^*)L(F)\phi(ee^*)$ and let $K_1^*$ be unitary $K_1$ as defined in \cite{classinvo}*{Equation (10.1)}. Put 
\[
\cK=\ker(K_0:kk(L(E),L(F))\to \hom(\BF(E),\BF(F)))
\]
We have a commutative diagram
\[
\xymatrix{K_1(L(F)_\phi)^*\ar[r]^*+[o][F-]{1}_{\sim}\ar[d]&(K_1(L(F))^*)^{E^1}\ar[d]^*+[o][F-]{2}&\\
K_1(L(F)_\phi)\ar[r]^{\sim}&K_1(L(F))^{E^1}\ar@{->>}[d]^*+[o][F-]{3}\ar@{->>}[r]&\cK\ar[d]\\
&\ker(I-A_F^t)^{E^1}\ar@{->>}[r]^(.7)*+[o][F-]{4}&\cL}
\]
The map labelled $1$ and the parallel map below are isomorphisms by Morita invariance \cite{classinvo}*{Lemma 10.2}. By \cite{cm2}*{Proposition 4.6}, the composite of the maps labelled $2$ and $3$ is onto. Moreover, because $F$ is spi, then by \cite{classinvo}*{Lemma 8.7}, $L(F)_\phi$ is strictly properly infinite in the sense of
\cite{classinvo}*{Section 8}
and thus, by \cite{classinvo}*{Proposition 10.4} $K_1(L(F)_\phi)^*$ is a quotient of the unitary group of $L(F)_\phi$. Summing up, there is a unitary $u\in \cU(L(F)_\phi)$ whose class in $K_1(L(F)_\phi)^*$ maps to $\eta$ through the composite of the maps labelled $1$ through $4$. By \cite{classinvo}*{Lemma 13.7} applied to $R=L(F)\oplus L(F)$ equipped with the involution $(x,y)^*=(y^*,x^*)$, the sum of the image of $u$ in $kk(L(E),L(F))$ with the class of $\phi$ is the class of the $*$-homomorphism $\phi^u:L(E)\to L(F)$ that maps $e\mapsto u\phi(e)$ which again has property (P). By construction, $k(\widehat{\phi^u})=\xi$. This proves part i). Let $\phi:L(E)\to L(F)$ be a $*$-homomorphism with property (P) such that $k(\hat{\phi})$ is an isomorphism. Then $j(\phi)$ is an isomorphism by parts i) and iii) of Theorem \ref{thm:onto} and therefore $\phi$ is an $M_2$-homotopy equivalence by \cite{cm2}*{Theorem 5.8}. If $\phi$ is unital and a polynomial $M_2$-homotopy equivalence, then it is a homotopy equivalence, by the argument at the end of the proof of \cite{cm2}*{Theorem 6.1}. This concludes the proof.
\end{proof}
\begin{coro}\label{coro:semimain} If $E$ and $F$ are finite spi graphs and either of $\BF(E)$, $\BF(F)$ is finite, then for every $*$-homomorphism $\xi:L(E)\to L(F)$ there exists a $*$-homomorphism $\phi:L(E)\to L(F)$ such that $\phi\sim_{M_2}\xi$. If moreover $\xi$ is unital, then $\phi$ can be chosen unital and so that $\phi\sim \xi$. 
\end{coro}
\begin{proof} Immediate from Corollary \ref{coro:onto} and Theorem \ref{thm:semimain}. 
\end{proof}

\begin{ex}\label{ex:splice} Let $\cR_n$ be the graph consisting of $1$ vertex and $n$ loops and $\cR_n^{-}$ its Cuntz' splice \cite{flow}*{Definition 2.11}. Write $L_n=L(\cR_n)$, $\cO_n=C^*(\cR_n)$, $L_{n^-}=L(\cR_{n^-})$ and $\cO_{n^-}=C^*(\cR_{n^-})$. Then by Theorem \ref{thm:semimain} and Corollary \ref{coro:onto}, every $*$-isomorphism
$\cO_n\iso\cO_{n-}$ is homotopic to the completion of a unital $*$-homomorphism $\phi:L_n\to L_{n^-}$, any such $\phi$ is an algebraic homotopy equivalence, and its homotopy class depends only on the $C^*$-homotopy class of $\xi$. 
\end{ex}

\appendix

\section{The \topdf{$M_2$}{M2}-homotopy relation}\label{sec:M2}

\begin{lem}\label{lem:abc}
Let $\fA,\fB\subset \fC$ be $C^*$-algebras and let $\inc_{\fA}$ and $\inc_{\fB}$ the inclusion maps. Let $x\in \fC$ such that $x\fA x^*\subset \fB$ and $ax^*xa'=aa'$ for all $a,a'\in \fA$. Then $\ad(x):\fA\to \fB$, $\ad(x)(a)=xax^*$ is a $*$-homomorphism and $\inc_{\fB}\circ\ad(x)\approx_{M_{2}}\inc_{\fA}$. If moreover $\fA=\fB$ and $\fA x\subset \fA$, then $\ad(x)\approx_{M_{2}}\id_{\fA}$. 
\end{lem} 
\begin{proof} This is the $C^*$-algebra version of \cite{classinvo}*{Lemma 8.12}; it is proved in the same way, using the fact that the two canonical corner inclusions of a $C^*$-algebra into the $2\times 2$ matrices are $C^*$-homotopic. 
\end{proof} 

Let $\fB$ be a properly infinite $C^*$-algebra and $s=(s_1,s_2)\in\fB^2$ a pair of orthogonal isometries. Then the map
\begin{equation}\label{map:boxplus}
\boxplus_s:\fB\oplus\fB\to \fB,\, (a,b)\mapsto s_1as_1^*+s_2bs_2^*
\end{equation}
is a $C^*$-algebra homomorphism, and if $\fA\triqui\fB$, it restricts to a homomorphism $\fA\oplus\fA\to\fA$.

\begin{lem}\label{lem:boxsum}
Let $\fA$, $\fB$ and $\fC$ be $C^*$-algebras with $\fC$ properly infinite and $\fB\triqui\fC$ an ideal. Let $t_1,t_2\in\fC$ be orthogonal isometries. Then $\boxplus_t$ makes $[[\fA,\fB]]_{M_2}$ into an abelian monoid.
\end{lem}
\begin{proof} Straightforward from Lemma \ref{lem:abc}.
\end{proof}
\begin{lem}\label{lem:samebox}
Let $\fB$ be a properly infinite $C^*$ algebra and let $s=(s_1,s_2)$ and $t=(t_1,t_2)$ such that $s_i^*s_j=t_i^*t_j=1$. Let $\fA\triqui\fB$ be a closed ideal. Then $\boxplus_s,\boxplus_t:\fA\oplus\fA\to\fA$ are $M_2$-homotopic. 
\end{lem}
\begin{proof} The element $u=t_1s_1^*+t_2s_2^*$ is a partial isometry with domain projection $p=\boxplus_s(1,1)=\sum_{i=1}^2s_is_i^*$. Then by Lemma \ref{lem:abc}, $\ad(u):p\fA p\to \fA$ is $M_{2}$-homotopic to the inclusion $p\fA p\subset \fA$. Hence $\boxplus_s\approx_{M_{2}}\ad(u)\circ\boxplus_s=\boxplus_t$. 
\end{proof}

Let $\fA$ and $\fB$ be $C^*$-algebras. By \cite{rostro}*{Section 8.1 and formula (8.1.3)}, the set
\[
[[\fA,\cK\sotimes\fB]]
\]
has a natural monoid structure. 

\begin{lem}\label{lem:m2htpy} Let $\fA$ and $\fB$ be $C^*$-algebras, with $\fB$ properly infinite. Let $p\in\cK$ be a rank one projection; set $\iota:\fB\to \cK\sotimes\fB$, $\iota(b)=p\sotimes b$. Then

\item[i)]$\iota$ induces an isomorphism of monoids
\[
\iota_*:[[\fA,\fB]]_{M_2}\iso[[\fA,\cK\sotimes\fB]]_{M_2}.
\]
\item[ii)] The canonical surjection is a monoid isomorphism
\[
[[\fA,\cK\sotimes\fB]]\iso [[\fA,\cK\sotimes\fB]]_{M_2}.
\]
Moreover both isomorphisms above restrict to bijections between the subsets of homotopy classes of full homomorphisms.
\end{lem}
\begin{proof}
It is clear that $\iota_*$ is a monoid homomorphism. Let $s\in \fB^{1\times \infty}$ be a row matrix of orthogonal isometries. Then $\ad(s):M_\infty \fB\to\fB$, $\ad(s)(x)=sxs^*$ is a $*$-homomorphism and thus induces a $C^*$-algebra homomorphism $\cK\sotimes\fB\to \fB$. Composition with $\ad(s)$ defines a map $\ad(s)_*:[[\fA,\cK\sotimes\fB]]_{M_2}\to [[\fA,\fB]]_{M_2}$. It follows from Lemma \ref{lem:abc} that $\ad(s)\circ\iota$ and $\iota\circ\ad(s)$ are $M_2$-homotopic to the identity maps. This proves i). The surjection of part ii) is a monoid homomorphism by Lemma \ref{lem:samebox}; to prove that it is injective, it suffices to show that the composite of the corner inclusion $\cK\to M_2\cK$ with an isomorphism
$M_2\cK\cong\cK$ is homotopic to the identity. This follows from \cite{rostro}*{Proposition 1.3.4(i) and the implications (8.1.2)}. 
\end{proof}

\end{document}